\documentclass{amsart}
\usepackage{graphics,epic}
\usepackage{amsmath,amssymb, amsthm, bm,color}
\usepackage[all,2cell]{xy}

\newtheorem{theorem}{Theorem}[section]
\newtheorem*{theorem*}{Theorem}

\newtheorem{lemma}[theorem]{Lemma}
\newtheorem{proposition}[theorem]{Proposition}
\newtheorem{corollary}[theorem]{Corollary}
\newtheorem{conjecture}[theorem]{Conjecture}
\newtheorem*{conjecture*}{Conjecture}

\theoremstyle{definition}

\theoremstyle{remark}
\newtheorem{remark}[theorem]{Remark}

\renewcommand{\hat}[1]{\widehat{#1}}

\newcommand{\al}{\alpha}

\newcommand{\Aut}{{\rm Aut}\,}

\newcommand{\Z}{\mathbb{Z}}

\newcommand{\C}{\mathbb{C}}

\newcommand{\la}{\lambda}

\def\C{{\mathbb C}}

\def\Z{{\mathbb Z}}

\def\1{{\bf 1}}
\def\la{{\langle}}
\def\ra{{\rangle}}

\def \pf{\noindent {\bf Proof: \,}}

\def \h{\mathfrak{h}}

\def \g{\mathfrak{g}}
\def \h{\mathfrak{h}}
\def \k{\mathfrak{f}}
\def \s{\mathfrak{s}}
\def \com{\mathrm{Com}}

\numberwithin{equation}{section}

\begin{document}

\title[Uniqueness of holomorphic VOAs]{$\mathbb{Z}_2$-orbifold construction associated with $(-1)$-isometry and uniqueness of holomorphic vertex operator algebras of central charge 24}

\author{Kazuya Kawasetsu}
\address{ Institute of Mathematics, Academia Sinica, Taipei 10617, Taiwan}
\email{kawasetsu@gate.sinica.edu.tw}
\author{Ching Hung Lam}
\address{Institute of Mathematics, Academia Sinica, Taipei 10617, Taiwan}
\email{chlam@math.sinica.edu.tw}
\author{Xingjun Lin}
\address{Institute of Mathematics, University of Tsukuba, Tsukuba, Ibaraki 305-8571, Japan}
\email{linxingjun88@126.com}

\thanks{C.\,H. Lam was partially supported by MoST grant 104-2115-M-001-004-MY3 of Taiwan}
\thanks{X. Lin is an "Overseas researchers under Postdoctoral Fellowship of Japan Society for the Promotion of Science", and is supported by JSPS Grant No. 16F16020.}

\begin{abstract}
The vertex operator algebra structure of a strongly regular holomorphic vertex operator algebra $V$ of central charge $24$ is proved to be uniquely determined by the Lie algebra structure of its weight one space $V_1$ if $V_1$ is a Lie algebra of the type
$A_{1,4}^{12}$, $B_{2,2}^6$,  $B_{3,2}^4$, $B_{4,2}^3$,  $B_{6,2}^2$,
$B_{12,2}$,
$D_{4,2}^2B_{2,1}^4$,  $D_{8,2}B_{4,1}^2$,
$A_{3,2}^4A_{1,1}^4$,  $D_{5,2}^2A_{3,1}^2$,  $D_{9,2}A_{7,1}$,
$C_{4,1}^4$
or $D_{6,2}B_{3,1}^2C_{4,1}$.
\end{abstract}
\subjclass[2010]{17B69}
\keywords{Vertex operator algebra, orbifold construction, Niemeier lattices}
\maketitle

\section{Introduction \label{intro}}

The classification of holomorphic vertex operator algebras (VOAs) of central charge $24$ is one of important problems in the theory of VOAs and conformal field theories.
In 1993, Schellekens~\cite{Sch} obtained a partial classification of holomorphic VOAs of central charge $24$
 and showed that there are $71$ possible Lie algebra structures for the weight one spaces of holomorphic VOAs of central charge $24$ (see also \cite{EMS}).
 Recently, holomorphic VOAs
of central charge $24$ corresponding to all $71$ Lie algebras in Schellekens' list
 have been explicitly constructed (see \cite{EMS}, \cite{LS15}, \cite{LLin} and~\cite{SS}).
To finish the classification of holomorphic VOAs
of central charge $24$, it remains to show that there is a unique  holomorphic VOA
of central charge $24$ corresponding to each Lie algebra in Schellekens' list.
Motivated by the fact that
the unimodular lattices of rank $24$ (Niemeier lattices) are determined by their root
systems, it is believed that the following conjecture is true.
\begin{conjecture}\label{conj}
The VOA structure of a~strongly regular holomorphic VOA $V$ of central charge
$24$ is uniquely determined by its weight one Lie algebra $V_1$.
\end{conjecture}


Up to now, Conjecture \ref{conj} has been verified in the followings cases:
(i)~the weight one Lie algebras of the $24$ Niemeier lattice VOAs (24 cases) \cite{DM04};
(ii)~$A_{1,2}^{16}$ and $E_{8,2}B_{8,1}$\cite{LS15a};
(iii)~$E_{6,3}G_{2,1}^3$, $A_{2,3}^6$ and $A_{5,3}D_{4,3}A_{1,1}^3$ \cite{LS16};
(iv)~$A_{8,3}A_{2,1}^2$ \cite{LLin}.
In this paper, we will consider $13$ other Lie algebras in Schellekens' list. More precisely, we will prove the following result.
\begin{theorem}\label{sec:cor1}
The structure of a~strongly regular holomorphic VOA $V$ of central charge $24$
is uniquely determined
by its weight one Lie algebra $V_1$ if $V_1$ has the type
\begin{align*}
&A_{1,4}^{12},\quad  B_{2,2}^6,\quad  B_{3,2}^4, \quad B_{4,2}^3,\quad  B_{6,2}^2,\quad  B_{12,2}, \quad
D_{4,2}^2B_{2,1}^4,\quad  D_{8,2}B_{4,1}^2, \\
&A_{3,2}^4A_{1,1}^4,\quad  D_{5,2}^2A_{3,1}^2,\quad  D_{9,2}A_{7,1}, \quad
C_{4,1}^4
\quad \mathrm{or}\quad
D_{6,2}B_{3,1}^2C_{4,1}.
\end{align*}
\end{theorem}
The holomorphic VOAs in Theorem \ref{sec:cor1} can be
 obtained by applying
$\Z_2$-orbifold construction to Niemeier lattice VOAs and
 lifts of the $(-1)$-isometry of the lattices \cite{DGM}. To apply  the ``Reverse orbifold construction" method proposed  in~\cite{LS16}, there are two key steps. The first step is to find an appropriate semisimple element $h\in V_1$ such that 
the VOA obtained by applying $\Z_2$-orbifold construction to $V$ and the inner automorphism $\sigma_h$ is isomorphic to a Niemeier lattice VOA $V_N$ (see Lemma \ref{mainLemma}). Since $\sigma_h$ acts trivially on the weight one subspace $V_1$ in our cases, the non-trivial part is to show that $\sigma_h$ has order $2$ on $V$ (cf. Lemma \ref{condition}). We also show a technical lemma (see Lemma~\ref{sec:difference}), which helps us to determine the lowest conformal weight of the irreducible twisted module and greatly reduces the amount of calculations  in our cases.  The second main step is to show that any order $2$ automorphism $\mu$ of the Niemeier lattice VOA satisfying $(V_{N})_1^\mu\cong (V_1)^{\sigma_h}$ is conjugate to $\theta$ (cf. Eq. \eqref{theta}).    Although such kinds of results are not easy to show in general, we manage to find an efficient way for proving them in our cases (see  Theorems \ref{main1} and \ref{main2}).

The following is the organization of the paper.
 In Section 2, we recall some facts about orbifold construction associated with inner automorphisms and reverse orbifold construction.  We also prove several lemmas which will be used to determine the lowest conformal weights of twisted modules.
 In Section 3, we determine the conjugacy class of the automorphism $\theta$  of the Niemeier lattice VOA. In Section~4, we determine the appropriate  semisimple element $h\in V_1$ and then prove our main theorem.
\section{Prelimaries}

\subsection{Basic facts about VOAs}
In this subsection, we recall some basic facts about VOAs from \cite{DM04a, DM04, FLM}. A~VOA $V$ is called \emph{strongly regular} if $V$ is self-dual, rational, $C_2$-cofinite and of CFT-type (cf.~\cite{DM04a, DM04}). We call a VOA $V$ a \emph{holomorphic} VOA if $V$ is rational and  has a unique irreducible module up to isomorphisms.

Let $V=\bigoplus_{n=0}^\infty V_n$ be a strongly regular VOA.
Here $V_n$ is the subspace of $V$ of conformal weight $n\in\Z_{\geq 0}$.
It then follows that the weight one space $V_1$ is a~Lie algebra with respect to the bracket $[u, v]=u_{(0)}v$ for any $u, v\in V_1$,
where  $u_{(n)}:V\rightarrow V$ denotes the $n$-th product of $u$ in $V$ for each $n\in\Z$ (see \cite{DM04}). Moreover,
for any simple Lie subalgebra $\mathfrak{s}\subset V_1$, the sub\,VOA of $V$
generated by $\mathfrak{s}$ is isomorphic to the~affine VOA
$L_\mathfrak{s}(k,0)$ for some positive integer $k$ \cite{DM04}.
We then call $\mathfrak{s}$ a~simple Lie subalgebra of $V_1$ with  level $k$ and
write $\mathfrak{s}=\mathfrak{s}_k\subset V_1$. Assume further that $V$ is a holomorphic VOA of central charge $24$, we then have the following result.

\begin{proposition}\cite[Theorem 3, (1.1)]{DM04a}\label{sec:hvee}
Let $V$ be a holomorphic VOA of central charge $24$.
If the Lie algebra $V_1$ is neither $\{0\}$ nor abelian, then $V_1$ is semisimple,
and the conformal vectors of $V$ and the sub\,VOA generated
by $V_1$ are the same.
If $V_1$ is semisimple, then for any simple ideal $\s$ of $V_1$ with the
 level $k\in\Z_{>0}$, the identity
 $h^\vee/k=(\dim V_1-24)/24$
holds, where $h^\vee$ is the dual Coxeter number of $\s$.
\end{proposition}

It is also known \cite{DM04} that there exists a unique symmetric invariant bilinear form $\la \cdot |\cdot\ra$ on $V$ such that $\la \1| \1\ra=-1$, 
where $\1$ is the vacuum vector of $V$. Furthermore,
for each simple Lie subalgebra $\s$ of $V_1$ with the level $k$,
we have  $\langle\cdot|\cdot\rangle|_\s=k(\cdot|\cdot)_\s$, where $(\cdot|\cdot)_\s$ denotes the normalized Killing form of $\s$ (see \cite{LS15}).
\vskip.25cm
Let  $R$  be a sub\,VOA of $V$. We consider the commutant $\com_V(R)$ of $R$ in $V$,
that is, $\com_V(R)=\{v\in V\,|\, w_{(n)}v=0, w\in R, n\geq 0\}$. 
\begin{lemma}{\rm(}\cite[Theorem~2]{KM}{\rm)}\label{sec:km}
Suppose that both $R$ and $\com_V(R)$ are strongly regular VOAs and satisfy $\com_V(\com_V(R))=R$.
Then any irreducible $R$-module is embedded in some
irreducible $V$-module as an~$R$-submodule.
\end{lemma}

We also need the following result.
\begin{lemma}{\rm(\cite[Theorem 3.5]{HKL15})}\label{sec:hkl}
Let $T$ be a~$C_2$-cofinite, simple VOA
of CFT-type and $S$ a~full sub VOA of $T$.
Assume that $S$ is  strongly regular and that the lowest conformal
 weight of any irreducible $S$-module is positive
except for the vacuum module of $S$.
Then $T$ is rational.
\end{lemma}

\subsection{Orbifold construction associated with inner automorphisms}\label{formula}
 In this subsection, we recall from \cite{LS15} some formulas about orbifold construction of holomorphic VOAs of central charge 24.

Let $V$ be a strongly regular holomorphic VOA of central charge $24$, $g$ be an~automorphism of $V$ of prime order $p$.
 We then know that there is a~unique $g^r$-twisted $V$-module $V^{\mathrm{T}}(g^r)$
for each $1\leq r\leq p-1$ (\cite[Theorem~10.3]{DLM00}).
Moreover, the fixed point subspace $V^g$ of $V$ with respect to $g$
is a~sub\,VOA of $V$.
The weight $n$ subspace of $V^g$ coincides with $V_n^g=V_n\cap V^g$ ($n\geq 0$).
We say that the pair $(V,g)$ satisfies the {\it orbifold condition} if
 there exists a~unique simple VOA $\tilde{V}$
 such that $V^g$ is embedded in $\tilde{V}$ and
$\tilde{V}\cong V^g\oplus \bigoplus_{r=1}^{p-1}V^{\mathrm{T}}(g^r)_\mathbb{Z}$ as a~$V^g$-module,
where $V^{\mathrm{T}}(g^r)_\Z$  is the subspace of $V^{\mathrm{T}}(g^r)$ of integral conformal weights (cf.~\cite{EMS}).
If $(V,g)$ satisfies the orbifold condition,  the VOA
$\tilde{V}$ which satisfies
the above assumptions is strongly regular and holomorphic. We refer to $\tilde{V}$ as the VOA obtained by applying the $\Z_p$-orbifold
construction to $V$ and $g$,  and denote the VOA $\tilde{V}$ by $\tilde{V}(g)$.

Suppose that the Lie algebra $V_1$ is semisimple. Then, $V_1$ is isomorphic to $ \g=\g_{(1),k_1}\oplus \cdots \oplus \g_{(t),k_t}$ for some simple ideals $ \g_{(1)},\cdots, \g_{(t)}$ with levels $k_1,\cdots,k_t\in\Z_{>0}$, respectively. Fix a Cartan subalgebra $\h$ of $V_1$, and let $h$ be a semisimple element in $\h$ such that:

\smallskip\noindent
(i): ${\rm Spec}(h_{(0)})\subset (1/2)\Z$ and
$\mathrm{Spec}(h_{(0)})\not\subset \Z$;

\noindent
(ii): $\la h| h\ra\in \Z$;

\noindent
(iii): The lowest conformal weight of $V^{(h)}$ is positive.

\smallskip\noindent
Then the inner automorphism $\sigma_h:=\exp(-2\pi\sqrt{-1}h_{(0)})$ of $V$ is of order $2$.

\begin{theorem}{\rm (\cite{LS15})}\label{existence}
Let $V$ and $h$ be as above. Then $(V,\sigma_h)$ satisfies the orbifold condition.
\end{theorem}

Moreover, we have the following result.
\begin{proposition}{\rm (\cite{Mo} and \cite{LS15})}
Let $V$, $h$ be as above. Then we have
$$
\dim V_1+\dim \tilde{V}(\sigma_h)_1=3\dim V^{\sigma_h}_1+24(1-\dim V^{\mathrm{T}}({\sigma_h})_{1/2}).
$$ 
In particular, if $V^{\sigma_h}_1=V_1$ and $V^{\mathrm{T}}({\sigma_h})_{1/2}=0$, then we have
\begin{equation}\label{eqn:montague}
\dim \tilde{V}(\sigma_{h})_1=2\dim V_1+24.
\end{equation}
\end{proposition}

\subsection{Reverse orbifold construction of holomorphic VOAs}\label{reverse}
In this subsection, we recall from \cite{LS16} the method called ``reverse orbifold construction". Let $V$ be a~strongly regular holomorphic VOA
and $g$ an~automorphism of $V$ of prime order $p$ such that $(V,g)$ satisfies the orbifold condition.
Let $$W=\tilde{V}(g)=V^g\oplus \bigoplus_{r=1}^{p-1}V^{\mathrm{T}}(g^r)_\Z$$ be the VOA
 obtained by applying the $\Z_p$-orbifold construction to $V$ and $g$.
Define an~automorphism $a=a_{V,g}$ of $W$ by
$a|_{V^{g}}=1$ and $a|_{V^{\mathrm{T}}(g^r)_\Z}=e^{2\pi\sqrt{-1}r/p}$ ($1\leq r\leq p-1$).
 It then follows that the pair $(W,a)$ satisfies the orbifold condition
and  $\tilde{W}(a)\cong V$ (see \cite{EMS}).

Let $\g$ be a~semisimple Lie algebra and $h$ a~semisimple element of $\g$.
Assume that there exists a~strongly regular holomorphic VOA $U$ such that
for any strongly regular holomorphic VOA $V$ satisfying $V_1\cong \g$, the following conditions hold:

\smallskip\noindent
{\rm (a):} $\sigma_h$ has prime order $p$ on $V$ and the pair $(V,\sigma_h)$ satisfies
the orbifold condition;

\noindent
{\rm (b):} $\tilde{V}(\sigma_h)\cong U$;

\noindent
{\rm (c):} for any automorphism $g$ of $U$ of order $p$, if $U^g_1\cong \g^{\sigma_h}$,
then $g$ is conjugate to  the automorphism $a_{V,\sigma_h}$
of $\tilde{V}(\sigma_h)\cong U$  in $\Aut(U)$.

\smallskip\noindent
Then we have the following result which was essentially obtained in \cite{LS16}.
\begin{theorem}
\label{sec:unique1}
The structure of a~strongly regular holomorphic VOA $V$
such that $V_1\cong \g$ is unique up to isomorphisms.
\end{theorem}

\begin{proof}
Let $V$ and $W$ be strongly regular holomorphic VOAs such that $V_1\cong \g\cong W_1$.
By the condition (b), we see that $\tilde{V}(\sigma_h)\cong U\cong \tilde{W}(\sigma_h)$.
Let $a$ and $b$ be the automorphisms of $U$ induced from the automorphisms
$a_{V,\sigma_h}$ and $a_{W,\sigma_h}$ of $\tilde{V}(\sigma_h)$ and  $\tilde{W}(\sigma_h)$,
respectively.
It then follows from (c) that $a$ is conjugate to $b$.
By applying the $\Z_p$-orbifold construction to $(U,a)$ and $(U,b)$,
we see that $V\cong W$.
\end{proof}

\begin{remark}
Although condition (c) in Theorem~\ref{sec:unique1} is sufficient for our
purpose in this paper,
it is usually too strong. Note that there exist automorphisms
$g$ and $h$ of a~lattice VOA $V_L$ such that
$(V_L)^g_1\cong (V_L)^h_1$ as Lie algebras but
$g$ and $h$ are not conjugate in $\Aut(V_L)$ (see for example~\cite[p.1583]{LS16a}).
In this situation, we may replace~(c) by a weaker condition

\smallskip\noindent
(c$'$): any automorphism $g$ of $U$ of order $p$
such that the pair $(U,g)$ satisfies
the orbifold condition and $\tilde{U}(g)_1\cong \g$ is conjugate to $a_{U,g}$.

\smallskip\noindent
Then Theorem~\ref{sec:unique1} still holds.
\end{remark}

In Section \ref{sec}, we will study the case that $V$ is a holomorphic VOA of central charge $24$ with the Lie algebra $V_1=\mathfrak{g}$, where $\g$ is one of Lie algebras in Table \ref{table1}. In this case, the VOA $U$ will be the lattice VOA $V_N$ associated to some Niemeier lattice $N$. To verify the condition (b) in this case, we will need the following result, which can be deduced from~\cite[Theorem~3]{DM04}.

\begin{proposition}\label{sec:Niemeier}
Let $N$ be a Niemeier lattice and $U$ be a strongly regular holomorphic VOA
 of central charge $24$
such that $U_1\cong (V_N)_1$. Then the vertex operator algebra $U$ is isomorphic to the lattice VOA
$V_{N}$.
\end{proposition}


We next recall some facts about automorphisms of lattice VOAs. Let $L$ be a positive definite even lattice. Set $\h=\C\otimes_{\Z}L$ and view $\h$ as an abelian Lie algebra equipped with a non-degenerate symmetric bilinear form. Let $\hat \h$ be the corresponding Heisenberg Lie algebra and $M(1)$ the highest weight module of $\hat \h$ with the highest weight $0$ (cf. \cite{FLM}). Then the lattice VOA $V_L$ associated to $L$ is defined on the vector space $M(1)\otimes \C[L]$, where $\C[L]={\rm span}~\{e^x|x\in L\}$ (cf. \cite{FLM}). In particular, $V_L$ is spanned by the vectors of the form $h^1(-n_1)\cdots h^k(-n_k)\otimes e^x$, where $h^1, \cdots, h^k\in \h$, $x\in L$ and $n_1, \cdots, n_k$ are positive integers (cf. \cite{FLM}).
Let $\theta:V_L\to V_L$ be the linear map determined by
\begin{align}\label{theta}
h^1(-n_1)\cdots h^k(-n_k)\otimes e^x&\mapsto (-1)^kh^1(-n_1)\cdots h^k(-n_k)\otimes e^{-x}.
\end{align}
It was proved in \cite{FLM} that $\theta$ is an automorphism of $V_L$ of order 2. Notice that $\theta$ is a lift of the $(-1)$-isometry of $L$ (cf. \cite{DGH}).

To determine the conjugacy class of the automorphism $\theta$, we also need the following result.
\begin{proposition}[{\cite[Theorem~D.6]{DGH}}]\label{sec:lift}%
Any lifts of the $(-1)$-isometry of $L$ are conjugate under $\mathrm{Aut}(V_L)$.
\end{proposition}

\section{Uniqueness of automorphisms of Lie algebras}
\subsection{Automorphisms of Lie algebras}
In this subsection, we recall some facts about automorphisms of simple Lie algebras from \cite{H} and \cite{DGM}.
Let $\mathfrak{s}$ be a~finite dimensional  simple Lie algebra of rank $n$ with
a Cartan subalgebra $\h$ of $\mathfrak{s}$.
For automorphisms $g$ and $g'$ of $\s$, we write $g \sim g'$ if $g$ is conjugate to $g'$ in $\Aut(\s)$.
Let $[x]$ denote the maximum integer less than or equal to a~real number $x$.
The following proposition can be found in \cite[Theorem~6.1, TABLE II and pp.513--515]{H}.

\begin{proposition}\label{sec:classify1}
Let $\sigma_1$ and $\sigma_2$ be automorphisms of $\mathfrak{s}$ of order $2$. Then $\sigma_1 \sim \sigma_2$ if and only if the Lie algebra $\mathfrak{s}^{\sigma_1}$
is isomorphic to $\mathfrak{s}^{\sigma_2}$.
Moreover, if $\sigma$ is an automorphism of $\s$ of order $2$
and $\s^\sigma$ is semisimple,
then the fixed point Lie algebra $\mathfrak{s}^\sigma$ is given by the following: 

\noindent 
{\bf (1)} $(A_{2n})^\sigma\cong B_n$ {\rm(}$n\geq 1${\rm)},
{\bf (2)} $(A_{2n+1})^\sigma\cong C_{n+1}$ or $D_{n+1}$ {\rm(}$n\geq 2${\rm)},
{\bf (3)} $(B_n)^\sigma\cong B_{n-p}\oplus D_p$ {\rm(}$n\geq 3$, $2\leq p\leq n${\rm)},
{\bf (4)} $(C_n)^\sigma\cong C_p\oplus C_{n-p}$, {\rm(}$n\geq 2$, $1\leq p\leq [n/2]${\rm)},
{\bf (5)} $(D_{n})^\sigma\cong D_p\oplus D_{n-p}$ {\rm(}$n\geq 4$, $2\leq p\leq [n/2]${\rm)}
and $(D_n)^\sigma\cong B_p\oplus B_{n-p-1}$
{\rm(}$n\geq 3$, $0\leq p\leq [(n-1)/2]${\rm)},
{\bf (6)} $(E_6)^\sigma\cong F_4$, $C_4$ or $A_1\oplus A_5$,
{\bf (7)} $(E_7)^\sigma\cong A_7$ or $A_1\oplus D_6$,
{\bf (8)} $(E_8)^\sigma\cong D_8$ or $A_1\oplus E_7$,
{\bf (9)} $(F_4)^\sigma\cong B_4$ or $A_1\oplus C_3$,
{\bf (10)} $(G_2)^\sigma\cong A_1\oplus A_1$.
\end{proposition}

Let $\theta\in\mathrm{Aut}(\mathfrak{s})$ be a~lift of the $(-1)$-automorphism of $\mathfrak{h}$. Then we have
\begin{proposition}\label{sec:theta} {\rm(cf.~\cite{DGM})}
If $\mathfrak{s}$ is simply-laced, then the fixed point Lie algebra $\mathfrak{s}^\theta$ is given by the following:\\
{\bf (1)} $(A_{2n})^\theta\cong B_n$,
{\bf (2)} $(A_{2n+1})^\theta\cong D_{n+1}$,
{\bf (3)} $(D_{2n})^\theta\cong D_n^2$,
{\bf (4)} $(D_{2n+1})^\theta \cong B_n^2$,
{\bf (5)}~$(E_6)^\theta\cong C_4$,
{\bf (6)} $(E_7)^\theta\cong A_7$,
{\bf (7)} $(E_8)^\theta\cong D_8$.
\end{proposition}
\smallskip

\subsection{Uniqueness of automorphisms of Lie algebras}\label{sec:uniqueLie}
In this subsection, we will prove that the conjugacy class of the automorphism $\theta$  of the Niemeier lattice VOA is uniquely determined by the Lie algebra structure of the fixed-point weight one subspace. This will be used to verify the condition (c) in Subsection \ref{reverse} in the proof of Theorem \ref{sec:cor1}. 

Let $V$ be a strongly regular simple VOA and $g$ an automorphism of $V$ of order $2$.
Assume that the Lie algebra $V_1$ is semisimple and let $V_1=\bigoplus_{i=1}^p\mathfrak{s}_{(i),\ell_i}$ be the
decomposition of $V_1$ into the sum of simple ideals $\mathfrak{s}_{(1)},\ldots,\s_{(p)}$ with levels $\ell_1,\ldots,\ell_p$, respectively.
Then $g$ acts on $\{\mathfrak{s}_{(i)}\,|\,1\leq i\leq p\}$ as a permutation.
Without loss of generality, we may assume that there exists a non-negative integer $q$ such that $2q\leq p$,
$g(\mathfrak{s}_{(i)})=\mathfrak{s}_{(i+q)}$ if $1\leq i\leq q$,
$g(\mathfrak{s}_{(i)})=\mathfrak{s}_{(i-q)}$ if $q+1\leq i\leq 2q$ and
$g(\mathfrak{s}_{(i)})=\mathfrak{s}_{(i)}$ if $2q+1\leq i \leq p$.
The following result can be established by the same argument as in~\cite{LS16}.

\begin{proposition}[cf.~{\cite[Proposition 3.7]{LS16}}]\label{sec:decomp1}
The fixed point Lie algebra $V_1^g$ is isomorphic to a sum of ideals
$$
\bigoplus_{i=1}^q(\mathfrak{s}_{(i)}\oplus \mathfrak{s}_{(i+q)})^g\oplus \bigoplus_{i=2q+1}^p\mathfrak{s}_{(i)}^g.
$$
For any $1\leq i\leq q$,
we have $\ell_i=\ell_{i+q}$.
In addition, the Lie subalgebra $(\mathfrak{s}_{(i)}\oplus \mathfrak{s}_{(i+q)})^g$ is
a~simple ideal of $V_1^g$ isomorphic to $\mathfrak{s}_{(i)}$
and its level is $2\ell_i$.
\end{proposition}

Let $\mathfrak{f}$ be a~semisimple Lie algebra with the decomposition
$\k=\bigoplus_{i=1}^t\k_{(j)}$ into simple ideals.
Let $\sigma$ be an automorphism of $\mathfrak{f}$ of order $2$ and
$\mathfrak{f}^\sigma=\bigoplus_{i=1}^s \mathfrak{g}_{(i)}$ the decomposition of
$\k^\sigma$ into simple ideals. 

\begin{corollary}\label{sec:decomp2}
If $\mathfrak{f}_{(j)}\not \cong \mathfrak{g}_{(i)}$ for all $1\leq j\leq t$ and $1\leq i\leq s$,
then $\mathfrak{f}^\sigma$ decomposes into the sum of ideals $\mathfrak{f}^\sigma=\bigoplus_{j=1}^t\mathfrak{f}_{(j)}^\sigma$.
\end{corollary}
\begin{proof}
Let $V=\bigotimes_{j=1}^t L_{\mathfrak{f}_j}(1,0)$.
It then follows that $\sigma$ induces an~order $2$ automorphism $g$ of $V$ such that
$g|_{V_1}=\sigma$.
Since $V$ is a strongly regular VOA  and $V_1\cong \mathfrak{f}$, we can get the result by Proposition~\ref{sec:decomp1}.
\end{proof}

We are now ready to describe our  main result in this subsection.
We consider a pair of Lie algebras $(\g, \k)$ as in Table \ref{table1}.


\begin{table}[bht]
\caption{Lie algebras $(\g, \k)$.}\label{table1}
\begin{tabular}{|l|c|c|c|}
\hline
Cases & $\g$ & $\k$ & \\ \hline
(A) ($n|12$) & $B_{n,2}^{12/n}$ & $A_{2n,1}^{12/n}$ &
\\
(B) ($n|4$)  & $D_{2n,2}^{4/n}B_{n,1}^{8/n}$   &  $A_{4n-1,1}^{4/n}D_{2n+1,1}^{4/n}$ &
\\ 
(C) ($n|4$)  & $D_{2n+1,2}^{4/n}A_{2n-1,1}^{4/n}$ & $A_{4n+1,1}^{4/n}X_{n,1}$ &
($X_1= D_{4}$, $X_2=D_{6}$, $X_4=E_{7})$ \\ 
(D) & $C_{4,1}^4$ & $E_{6,1}^4$ & \\ 
(E) & $D_{6,2}B_{3,1}^2C_{4,1}$ & $A_{11,1}D_{7,1}E_{6,1}$& \\ \hline
\end{tabular}
\medskip

Here $n|N$ denotes the condition that the positive integer $n$ divides $N$. We also use the identifications:   $D_{2,k}=A_{1,k}^2$, $D_{3,k}=A_{3,k}$ and $B_{1,k}=A_{1,2k}$.
\end{table}

The following is the main result of this subsection. 

\begin{theorem}\label{main1}
Let $(\mathfrak{g}, \mathfrak{f})$ be a pair of semisimple Lie algebras listed in Table \ref{table1}. Then any automorphism $\sigma$ of $\k$ of order $2$ such that $\mathfrak{f}^\sigma\cong \mathfrak{g}$
is conjugate to $\theta$ in $\Aut(\mathfrak{f})$.
\end{theorem}

\pf
By using Corollary~3.4, we see that $\k^\sigma=\bigoplus_{j=1}^t \k_{(j)}^\sigma$.
We prove the assertion in a~case-by-case basis.

\noindent
{\bf (A)} Let $(\g,\k)=(B_{n,2}^{12/n}, A_{2n,1}^{12/n})$,
where $n\in\mathbb{Z}_{>0}$ and $n|12$.
 It then follows by Propositions~\ref{sec:classify1} and~\ref{sec:theta} that
 $A_{2n}^\sigma \cong B_n$ and $\sigma$ is conjugate to $\theta$.

\noindent
{\bf (B)} Let $(\g,\k)=(D_{2n,2}^{4/n}B_{n,1}^{8/n},A_{4n-1,1}^{4/n}D_{2n+1,1}^{4/n})$,
 where $n\in\mathbb{Z}_{>0}$ and $n|4$.
 If $n=1$, it follows that $\g=A_{1,2}^{16}$ and $\k=D_3^8=A_3^8$.
 By Propositions~\ref{sec:classify1} and~\ref{sec:theta}, we have $\k^\sigma\cong A_{1,2}^{16}$ and $\sigma\sim \theta$.
Suppose that $n=2$ or $4$.
 By Proposition~\ref{sec:classify1}, we see that $D_{2n+1}^\sigma\cong B_n^2$,
 and $\sigma|_{D_{2n+1}}$ is unique up to conjugate.
 Therefore, $(A_{4n-1}^\sigma)^{4/n}\cong D_{2n}^{4/n}$, and $\sigma|_{A_{4n-1}}$ is unique
 up to conjugate.
 Finally,  we have $\sigma\sim \theta$  by Proposition~\ref{sec:theta}.

 \noindent
 {\bf (C)} Let $(\g,\k)=(D_{2n+1,2}^{4/n}A_{2n-1,1}^{4/n},
 A_{4n+1,1}^{4/n}X_{n,1})$, where $n\in\mathbb{Z}_{>0}$, $n|4$,
 $X_1=D_4$, $X_2=D_6$ and $X_4=E_7$.
 It then follows that $A_{4n+1}^\sigma\cong D_{2n+1}$ and $\sigma|_{A_{4n+1}}\sim \theta$.
Moreover, $X_n^\sigma\cong A_{2n-1}^{4/n}$, and
 $\sigma\sim \theta$ by Propositions~\ref{sec:classify1} and~\ref{sec:theta}.

 \noindent
 {\bf (D)} Let $(\g,\k)=(C_{4,1}^4,E_{6,1}^4)$.
In a~similar way, we see that $E_6^\sigma=C_4$ and $\sigma|_{E_6}\sim\theta$. Hence, $\sigma$ is conjugate to $\theta$.

 \noindent
 {\bf (E)} Let $(\mathfrak{g},\mathfrak{f})=(D_{6,2}B_{3,1}^2C_{4,1}, A_{11,1}D_{7,1}E_{6,1})$.
By Proposition~\ref{sec:classify1}, we see that $A_{11}^\sigma\cong D_6$ and
$\sigma|_{A_{11}}$ is conjugate to $\theta$.
Similarly, $E_6^\sigma \cong C_4$ and $\sigma|_{E_6}\sim \theta$.
Therefore, we have $D_7^\sigma \cong B_3^2$, and hence $\sigma|_{D_7}\sim \theta$.
Thus, $\sigma$ is conjugate to $\theta$.
\qed

\vskip.25cm

Combining Proposition~\ref{sec:lift} and Theorem \ref{main1}, we immediately obtain the second main result in this subsection.
\begin{theorem}\label{main2}
Let $(\mathfrak{g}, \mathfrak{f})$ be one of the pairs of semisimple Lie algebras listed in Table \ref{table1}, $N(\k)$ be the Niemeier lattice such that $(V_{N(\k)})_1=\k$, $\theta$ be the automorphism of $V_{N(\k)}$ defined above. Then any automorphism $\mu$ of $V_{N(\k)}$ of order $2$
such that $(V_{N(\k)})_1^\mu\cong \g$
is conjugate to $\theta$ under $\Aut(V_{N(\k)})$.
\end{theorem}
\section{Uniqueness of holomorphic VOAs of central charge 24}\label{sec}

\subsection{Conformal weights of twisted modules of affine VOAs}
To apply the ``reverse orbifold construction" method on a holomorphic VOA $V$, we need to choose an appropriate semisimple element $h\in V_1$. One of the restrictions on $h$ is about the conformal weights of the irreducible $\sigma_h$-twisted $V$-modules. In this subsection, we will prove some results about  conformal weights of $\sigma_h$-twisted $V$-modules.

First, we recall some facts about simple Lie algebras. Let $\mathfrak{s}$ be a finite dimensional simple Lie algebra
and $\h$ a~Cartan subalgebra of $\s$
with the simple roots $\al_1, \dots, \al_n$ and fundamental weights
$\varpi_1,\ldots,\varpi_n$ labelled as in~\cite{Bou}.
The highest root of $\s$ is denoted by $\theta_0$.
Let $(\cdot|\cdot)$ be  the normalized Killing form of $\s$ so that $(\alpha|\alpha)=2$ for any long root $\alpha$.
We identify $\h$ and $\h^*$ via $(\cdot|\cdot)$.
A~vector $v\in \s$ has $\s$-weight $\lambda\in\h$ if
$[x,v]=(x|\lambda)v$ for any $x\in\h$, where $[\cdot,\cdot]$ is the Lie bracket of $\s$.
The set of the dominant integral weights of $\mathfrak{s}$ is denoted by $P^+(\s)$.
For any positive integer $k$, we denote by
$
P^+(\mathfrak{s},k)= \{ \lambda \in P^+(\s)\mid  ( \lambda|\theta_0) \leq k\}
$
the set of all dominant integral weights of $\mathfrak{s}$ with the level $k$.

For a~dominant integral weight  $\lambda$ of $\s$, let $L(\lambda)$ be the irreducible $\s$-module with highest weight $\lambda$.
We denote by $\Pi(\lambda)$ the set of all weights of $L(\lambda)$.
Let $i$ be a~node of the Dynkin diagram of $\s$.

\begin{lemma}\label{sec:min1}
{\bf(1)} If $\s$ is of type $D_{2n}$, then
$\min\{(\varpi_i|\mu)\,|\,\mu\in \Pi(\lambda)\}=-(\varpi_i|\lambda)$.
{\bf (2)} If $i$ is fixed by any diagram automorphism of the Dynkin diagram of $\s$,
then $\min\{(\varpi_i|\mu)\,|\,\mu\in \Pi(\lambda)\}=-(\varpi_i|\lambda)$.
\end{lemma}

\begin{proof}
Let $w_0$ be the longest element of the Weyl group of $\s$.
Since the lowest weight of $L(\lambda)$ is $w_0(\lambda)$ and $\varpi_i$ is a~dominant weight,
it follows that $\min\{(\varpi_i|\mu)\,|\,\mu\in \Pi(\lambda)\}=(\varpi_i|w_0(\lambda))$. In the case that $\s$ is of type $D_{2n}$, it is known that $w_0$ is equal to $-1$ (see \cite{Hu}), which shows (1).
Suppose that $i$ is fixed by any diagram automorphism of $\s$.
We see that the automorphism $-w_0$ is
(the standard lift of) a~diagram automorphism
as it permutes positive simple roots of $\s$ and preserves the inner product.
Since $w_0$ is an~involution,
it follows that $(\varpi_i|w_0(\lambda))=(w_0(\varpi_i)|\lambda)=-(\varpi_i|\lambda)$. Thus, we obtain (2), as desired.
\end{proof}

We next recall some facts about affine VOAs. Let $k$ be a positive integer and $L_{\s}(k,0)$ be the affine VOA associated to $\s$ with level $k$. It is known \cite{FZ} that $L_{\s}(k,0)$ is a~strongly regular VOA and
the set of all irreducible modules over $L_\s(k,0)$ up to isomorphisms
is given by $\{L_\s(k,\lambda)\,|\,\lambda\in P^+(\s,k)\}$,
where $L_\s(k,\lambda)$ is the irreducible $L_\s(k,0)$-module of $\s$-weight $\lambda$.

%


Consider the VOA $W=\bigotimes_{i=1}^t L_{\mathfrak{g}_{(i)}}(k_i,0)$, where $k_1, ..., k_t$ are positive integers.
Then any irreducible $W$-module is isomorphic to $\bigotimes_{i=1}^t L_{\mathfrak{g}_{(i)}}(k_i,\lambda_i)$ with $\lambda_i\in P^+(\mathfrak{g}_{(i)},k_i)$ for each $1\leq i\leq t$.
Let $h=(h_1,h_2,\ldots,h_t)$ be a~semisimple element of $W_1$ such that $(h|\alpha)\geq -1$ for any root $\alpha$ of $W_1$ and the spectrum ${\rm Spec}(h_{(0)})$ of $h_{(0)}:W\rightarrow W$
  is contained in $(1/T)\Z$ for some positive integer $T$. Then we know that $\sigma_h$ is an inner automorphism of $W$ such that $\sigma_{h}^T=1$. Moreover,
  for each $W$-module $M$, it is proved in \cite{Li} that $\left(M^{(h)}, Y_M^{(h)}(\cdot,z)\right):=\left(M, Y_M(\Delta(h, z)\cdot, z)\right)$ is a~$\sigma_h$-twisted $W$-module, where $Y_M(\cdot, z)$ is the vertex operator map of $M$ and
  $\Delta(h, z)=z^{h_{(0)}}\text{exp}\Bigl(\sum_{n=1}^{\infty}\frac{h_{( n)}}{- n}(-z)^{- n}\Bigr).$

\begin{lemma}[{\cite[Lemma 2.7]{LS16}}]\label{w(l)} Set $\bm{P}_{\g}=P^+(\mathfrak{g}_{(1)},k_1)\times \cdots\times P^+(\mathfrak{g}_{(t)},k_t)$ and let $\bm{\lambda}=(\lambda_1,\lambda_2,\ldots,\lambda_t)$ be an~element
of $\bm{P}_\g$.
Then the lowest conformal weight of $(\bigotimes_{i=1}^t L_{\mathfrak{g}_{(i)}}(k_i,\lambda_i))^{(h)}$ is equal to
$
w(\bm{\lambda})=\ell(\bm{\lambda})+\sum_{i=1}^t \min\{(h_i|\mu)\,|\,\mu\in \Pi(\lambda_i)\}+\langle h|h\rangle/2,
$
where $\ell(\bm{\lambda})$ is the lowest conformal weight of $\bigotimes_{i=1}^t L_{\mathfrak{g}_{(i)}}(k_i,\lambda_i)$ and $\Pi(\lambda_i)$ is the set of all weights of the irreducible
$\mathfrak{g}_{(i)}$-module $L(\lambda_i)$ with the highest weight $\lambda_i$.
\end{lemma}

We now let $V$ be a strongly regular, holomorphic VOA such that $V_1$ is semisimple, $h=(h_1,h_2,\ldots,h_t)$ be a~semisimple element of $V_1$ such that $(h|\alpha)\geq -1$ for any root $\alpha$ of $V_1$ and the spectrum ${\rm Spec}(h_{(0)})$ of $h_{(0)}:V\rightarrow V$
  is contained in $(1/T)\Z$ for some positive integer $T$. Then we know that $\sigma_h$ is an inner automorphism of $V$ of finite order. Assume that $V_1=\g\cong \g_{(1)}\oplus \cdots\oplus \g_{(t)}$ for some simple Lie algebras $\g_{(1)}, \cdots, \g_{(t)}$. For each $\bm{\lambda}\in \bm{P}_\g$, set $d(\bm{\lambda})=w(\bm{\lambda})-\ell(\bm{\lambda})$ .
 Write $\bm{0}=(0,0,\ldots,0)\in\bm{P}_\g$ and $L(\bm{\lambda})=\bigotimes_{i=1}^t L_{\mathfrak{g}_{(i)}}(k_i,\lambda_i)$.
We then have
\begin{lemma}\label{sec:difference}
Assume $d(\bm{\lambda})>-3/2$ for any $\bm{\lambda}\in \bm{P}_{\g}$ and $d(\bm{0})>1/2$.
Then the lowest conformal weight of $V^\mathrm{T}(\sigma_h)$ is greater than $1/2$. In particular, $V^\mathrm{T}(\sigma_h)_{1/2}=0$.
\end{lemma}

\begin{proof}
Note that the sub VOA of $V$ generated by $V_1$ is isomorphic to $\bigotimes_{i=1}^t L_{\mathfrak{g}_{(i)}}(k_i,0)$ for some positive integers $k_1,..., k_t$. Thus,  $V$ viewed as
a~$\bigotimes_{i=1}^t L_{\mathfrak{g}_{(i)}}(k_i,0)$-module has the decomposition
$V\cong \bigoplus_{j=0}^n L(\bm{\lambda}^{(j)}),$
where $n\geq 0$, $\bm{\lambda}^{(0)},\ldots,\bm{\lambda}^{(n)}\in \bm{P}_{\g}$ and $\bm{\lambda}^{(j)}=\bm{0}$ if and only if $j=0$.
It then follows that $V^T(\sigma_h)=\bigoplus_{j=0}^n L(\bm{\lambda}^{(j)})^{(h)}$ (see \cite{Li}).
Therefore, it suffices to show $w(\bm{\lambda}^{(j)})>1/2$ for all $0\leq j\leq n$.
By assumption, $d(\bm{0})>1/2$ and $\ell(\bm{0})=0$; hence, we have $w(\bm{0})>1/2$.
Assume that $0<j\leq n$.
Since $V_1\cong \mathfrak{g}$ and $\dim V_0=1$, we have $\ell(\bm{\lambda}^{(j)})\geq 2$.
It follows immediately from the assumption $d(\bm{\lambda}^{(j)})>-3/2$ that $w(\bm{\lambda}^{(j)})>1/2$. The proof is complete.
\end{proof}

\subsection{Orbifold construction of holomorphic VOAs}
\label{sec:applyorbifold}
In this subsection, we begin to prove Theorem \ref{sec:cor1}. To make the statement of Theorem \ref{sec:cor1} more precise, we will prove the following theorem. 

\begin{theorem}\label{mainresult}
Let $( \g, \k)$ be a pair of Lie algebras listed in Table~\ref{table1}. Let $V$ be a strongly regular holomorphic VOA
 of central charge $24$ such that $V_1$ is isomorphic to $\g$. Then $V$ is isomorphic to the VOA
$\tilde{V}_{N(\k)}(\theta)$,
where  $N(\k)$ is the Niemeier lattice such that $(V_{N(\k)})_1=\k$ and $\theta$ is the automorphism of the lattice VOA ${V}_{N(\k)}$ defined as in (\ref{theta}).
\end{theorem}

%

Note that Theorem \ref{sec:cor1} follows immediately from Theorem \ref{mainresult}. We will prove Theorem \ref{mainresult} after several lemmas. Our idea is to apply the ``reverse orbifold construction" method on the holomorphic VOA $V$.
We start by choosing an appropriate semisimple element $h\in \g$.  Let $\g=\g_{(1),k_1}\oplus \cdots \oplus \g_{(t), k_t}$ be a semisimple Lie algebra listed ,
where  $\g_{(i),k_i}$'s are arranged in the same order as in Table \ref{table1}.
For example, if $\g=D_{6,2}B_{3,1}^2C_{4,1}$, then
$\g_{(1)}=D_6$, $\g_{(2)}=B_3$, $\g_{(3)}=B_3$ and $\g_{(4)}=C_4$.

\begin{table}[bht]
\caption{Choice of $h$.}\label{table 2}
\begin{tabular}{|l|c|c|}
\hline
Cases & $\g$ & $h$ \\ \hline
(A) ($n|12$) & $B_{n,2}^{12/n}$ & $(\varpi_1,0,\ldots,0)$ ($12/n-1$ times $0$'s)\\
(B) ($n|4$)  & $D_{2n,2}^{4/n}B_{n,1}^{8/n}$   &  $(\varpi_1,0,\ldots,0)$ ($12/n-1$ times $0$'s)\\
(C) ($n|4$)  & $D_{2n+1,2}^{4/n}A_{2n-1,1}^{4/n}$ & $(0,\ldots,0,\varpi_n,\ldots,\varpi_n)$ ($4/n$ times $0$'s and $\varpi_n$'s)
\\
(D) & $C_{4,1}^4$ & $(\varpi_4,0,0,0)$\\
(E) & $D_{6,2}B_{3,1}^2C_{4,1}$ & $(\varpi_1,0,0,0)$ \\
\hline
\end{tabular}
\end{table}

\begin{lemma}\label{condition}
Let $V$ and $(\g, \k)$ be as above and let $h$  be the semisimple element of $\g$ defined as in Table~\ref{table 2}, where  $\varpi_1$ of $D_2$ means the weight $(\varpi_1,\varpi_1)$ of $A_1^2$. Then $h$ satisfies $\langle h|h\rangle=2$, $(h| \bm{\lambda}) \in \frac{1}2 \Z$, $d(\bm{0})=1$ and $d(\bm{\lambda})>-3/2$ for any $\bm{\lambda}\in\bm{P}_{\mathfrak{g}}$ such that $\bm{\lambda}\neq \bm{0}$. Moreover, $\sigma_h$ is an automorphism of $V$ of order $2$ such that $\g^{\sigma_h}=\g$.
\end{lemma}
\pf
 Since $h$ is a~sum of
fundamental weights corresponding to long roots,
it follows immediately that $\g^{\sigma_h}=\g$.
By direct calculations, it is also easy to verify that $\langle h|h\rangle=2$ and
$(h | \lambda)\in \frac{1}2\Z$ for all $\lambda\in P^+(\g_{(i)},k)$ ($1\leq i\leq t$).
Moreover, for any $\bm{\lambda}=(\lambda_1,\lambda_2,\ldots,\lambda_t)\in \bm{P}_{\g}$, we have
$d(\bm{\lambda})=
-\sum_{i=1}^{t/2}(\varpi_n|\lambda_{t/2+i})+1$
if $\g$ is in case (C), and
$d(\bm{\lambda})=
-(h|\lambda_1)+1$ otherwise,
by using Lemma \ref{sec:min1}. It is now straightforward to show that $d(\bm{0})=1$ and $d(\bm{\lambda})>-3/2$ for any $\bm{\lambda}\in\bm{P}_{\g}$ such that $\bm{\lambda}\neq \bm{0}$.

Finally, we show that the order of $\sigma_h$ is $2$.
Set  $\mathfrak{r}=\g_{(t/2+1),k_{t/2+1}}\oplus\cdots\oplus \g_{(t),k_t}$
and $\mathfrak{s}=\g_{(1),k_1}\oplus \cdots\oplus \g_{(t/2),k_{t/2}}$ when $\g$ is in case (C), and set $\mathfrak{r}=\g_{(1),k_1}$ and $\mathfrak{s}=\g_{(2),k_2}\oplus\cdots\oplus \g_{(t),k_t}$ otherwise.
Then $h$ belongs to $\mathfrak{r}$.
Let $R$ and $S$ be the sub\,VOAs of $V$ generated by $\mathfrak{r}$ and
$\mathfrak{s}$, respectively.
It follows that $R$ and $S$ are strongly regular.
We divide the proof into $3$ parts: ${\bf (1^\circ)}$ $\g\neq B_{12,2}$ and $D_{8,2}B_{4,1}^2$,
${\bf (2^\circ)}$ $\g=B_{12,2}$ and ${\bf (3^\circ)}$ $\g=D_{8,2}B_{4,1}^2$.

\medskip\noindent
 ${\bf (1^\circ)}$ Suppose that $\g\neq B_{12,2}$ and $D_{8,2}B_{4,1}^2$.
It then follows that
 the lowest conformal weight of any irreducible
$R$-module does not belong to $\mathbb{Z}_{\geq 2}$.
Therefore, we have $\com_V(\com_V(R))=R$.
Since $T=\com_V(R)$ is an~extension of  $S$,
the VOA $T$ is $C_2$-cofinite and of CFT-type.
Since the commutant of a~rational simple sub\,VOA in a~rational simple VOA is
also simple (\cite{ACKL}),
by Lemma~\ref{sec:hkl}, $T$ is rational.
By applying Lemma~\ref{sec:km} to $R$ and $T$,
we see that
all the irreducible modules of $R$
must appear in $V$.
Since there exists an~irreducible $R$-module
of $\mathfrak{r}$-weight $\lambda$
such that $(\lambda|h)\in1/2+\mathbb{Z}$,
 the order of $\sigma_h$ is $2$.

 \medskip\noindent
 ${\bf (2^\circ)}$ Suppose that $\g=B_{12,2}$.
Then as a~module of $R\cong L_{B_{12}}(2,0)$, $V$ decomposes as
$$
V\cong L_{B_{12}}(2,0)\oplus\bigoplus_{i=1}^{12}a_i L_{B_{12}}(2,\varpi_i)
\oplus \bigoplus_{i,j\in\{1,12\},\,i\leq j}b_{ij} L_{B_{12}}(2,\varpi_i+\varpi_j)
$$
with non-negative integers $a_i$ ($1\leq i\leq 12$) and $b_{ij}$ ($i,j\in\{1,12\}$, $i\leq j$).
By computing the lowest conformal weights of the irreducible
$R$-modules, we see that
\begin{equation}\label{eqn:dioph}
\dim V_2=\dim (L_{B_{12}}(2,0))_2+a_5 \dim L(\varpi_5)+ b_{1,12} \dim L(\varpi_1+\varpi_{12}).
\end{equation}
Here, $L(\lambda)$ is the irreducible $B_{12}$-module of highest weight $\lambda$.
Since $V$ is a~holomorphic VOA of central charge $24$, the
character of $V$ coincides with $j(\tau)-744+\dim B_{12}$, where $j(\tau)$ is the $j$-function.
Therefore, we have $\dim V_2=196884$.
We also have $\dim (L_{B_{12}}(2,0))_2=45450$,
$\dim L(\varpi_5)=53130$ and $\dim L(\varpi_1+\varpi_{12})=98304$.
It then follows by~\eqref{eqn:dioph} that $a_5=b_{1,12}=1$.
Since $(\varpi_1+\varpi_{12}|h)=3/2$, we see that the order of
$\sigma_h$ is $2$.

\medskip\noindent
 ${\bf (3^\circ)}$ Suppose that $\g=D_{8,2}B_{4,1}^2$.
 The set of all irreducible $R=L_{D_8}(2,0)$-modules $M$ such that
 the lowest conformal weight of $M$ belongs to $\mathbb{Z}_{\geq 2}$
 consists of $L_{D_8}(2,2\varpi_7)$ and $L_{D_8}(2,2\varpi_8)$.
 They are self-dual simple current modules such that $L_{D_8}(2,2\varpi_7)\boxtimes L_{D_8}(2,2\varpi_8)\cong L_{D_8}(2,2\varpi_1)$,
 where $\boxtimes$ denotes the fusion product of $R$-modules.
Therefore, we see that either (i) $\com_V(\com_V(R))=R$ or
(ii) $\com_V(\com_V(R))=R\oplus L_{D_8}(2,2\varpi_i)$ ($i=7,8$) holds.
 If (i) holds, then by a~similar argument to ($1^\circ$) above, we see that
 $L_{D_8}(2,\varpi_8)$ is a~summand
  in the decomposition of $V$ as a~$R$-module.
 Since $(\varpi_8|h)=1/2$, the order of $\sigma_h$ is $2$.
Suppose that (ii) holds with $i=7$ or $8$.
  It then follows by the theory of simple current extensions
 that the irreducible $\com_V(\com_V(R))$-modules are given by
$R\oplus L_{D_8}(2,2\varpi_i)$,
$L_{D_8}(2,2\varpi_1)\oplus L_{D_8}(2,2\varpi_j)$,
$L_{D_8}(2,\varpi_i)^\pm$,
$L_{D_8}(2,\varpi_1+\varpi_j)^\pm$,
$L_{D_8}(2,\varpi_2)\oplus L_{D_8}(2,\varpi_6)$,
$L_{D_8}(2,\varpi_4)^\pm$,
 where $j$ satisfies $\{i,j\}=\{7,8\}$.
 In particular, $L_{D_8}(2,\varpi_i)^+\subset V$ as a~module of
 $\com_V(\com_V(R))$.
 By a similar argument as in ($1^\circ$), we see that both $\com_V(R)$ and $\com_V(\com_V(R))$ are
 regular.
Since $(h|\varpi_i)=1/2$,  by applying Lemma~\ref{sec:km} to $\com_V(\com_V(R))$ and $\com_V(R)$,
 we have shown that the order of $\sigma_h$ is $2$.
\qed

\begin{lemma}\label{lemma2}
Let $V$, $(\g, \k)$ and $h$ be as in Lemma \ref{condition}. Then  $(V,\sigma_h)$ satisfies the orbifold condition,
and $V^{\mathrm{T}}(\sigma_h)_{\frac{1}2}=0$.
\end{lemma}

\pf By Lemmas \ref{condition} and \ref{sec:difference}, $h$ satisfies the conditions (i), (ii) and (iii) in Subsection \ref{formula}. Hence,  by Theorem \ref{existence}, $(V,\sigma_h)$ satisfies the orbifold condition.
Moreover, we have $V^{\mathrm{T}}(\sigma_h)_{1/2}=0$ by Lemma~\ref{sec:difference}.
\qed

By Lemmas \ref{condition} and \ref{lemma2}, we have the strongly regular holomorphic VOA $\tilde{V}(\sigma_h)$ of central charge $24$.
%

\begin{lemma}\label{mainLemma}
Let $V$, $(\g, \k)$ and $h$ be as in Lemma \ref{condition}.
Then the holomorphic VOA
 $\tilde{V}(\sigma_h)$ is isomorphic to $V_{N(\k)}$.
\end{lemma}

\pf
From now on, we set $W=\tilde{V}(\sigma_h)$.
Then since the Lie algebra $W_1$ is semisimple by Proposition~\ref{sec:hvee},
we have the decomposition $W_1=\bigoplus_{i=1}^r\mathfrak{s}_{(i),\ell_i}$ of
$W_1$ into simple ideals,
where $r\in\Z_{> 0}$ and $\mathfrak{s}_{(i)}$ is a~simple ideal of $W_1$ with the level $\ell_i\in\Z_{>0}$ ($1\leq  i\leq r$).
Let $a=a_{V,\sigma_h}$ be the automorphism of $W$ defined in Subsection~\ref{reverse}, which is of order $2$.
It then follows from $\g^{\sigma_h}=\g$ that $W_1^a\cong \g$.
We  give a case by case analysis to show the assertion.

\noindent
{\bf (A)} Let $(\g,\k)=(B_{n,2}^{12/n}, A_{2n,1}^{12/n})$,
where $n|12$.
Since $\dim\g=12 (2 n+1)$, it follows by~\eqref{eqn:montague}
 that $W_1$ has dimension $48(n+1)$.
Let $i$ be an~element of $\{1,\ldots,r\}$.
Then, by Proposition~\ref{sec:hvee}, we have
$h_i^\vee/\ell_i=(\dim W_1-24)/24=2n+1$, where $h^\vee_i$ denotes the dual Coxeter number of $\mathfrak{s}_{(i)}$.
Since $\ell_i$ is a~positive integer, $h_i^\vee$ is divisible by $2n+1$.
We now prove the assertion for each $n$.

\noindent
{\bf (A.1) Case of $n=1$.}
Since $\g=A_{1,4}^{12}$, by applying Propositions~\ref{sec:decomp1}
and~\ref{sec:classify1} to $W_1$ and $a\in\mathrm{Aut}(W_1)$,
we see that $\s_{(i)}$ is of type $A_1$, $A_2$, $B_3$, $C_2$, $D_4$, $D_3$ or $G_2$.
As $3|h_i^\vee$, we have
$\mathfrak{s}_{(i)}\cong A_2,C_2$ or $D_4$.
Since $\dim\,\g=96$,
$W_1$ is of type $D_{4,2}^2C_{2,1}^4$, $D_{4,2}A_{2,1}C_{2,1}^6$, $D_{4,2}A_{2,1}^6C_{2,1}^2$, $A_{2,1}^2C_{2,1}^8$, $A_{2,1}^7C_{2,1}^4$ or $A_{2,1}^{12}$.
Since $\g=A_1^{12}$, it follows by Proposition~\ref{sec:classify1} that
$\s_{(i)}^a\cong A_1^4, A_1^2$ and $A_1$ if $\s_{(i)}\cong D_4$,
$C_2$ and $A_2$, respectively.
As the Lie rank of $\g$ is $12$, we have $W_1\cong A_{2,1}^{12}$.
It then follows by Proposition~\ref{sec:Niemeier}
that
$W$ is isomorphic to $V_{N(\k)}$.

\noindent
{\bf (A.2) Case of $n=2$.}
Similarly, by Propositions~\ref{sec:decomp1} and~\ref{sec:classify1}, we have $\s_{(i)}\cong B_2$, $A_4$, $C_4$ or $D_5$.
Since $5|h_i^\vee$,
it follows that  $\s_{(i)}\cong A_4$ or $C_4$.
Therefore,
$W_1=C_{4,1}^4$, $C_{4,1}^2A_{4,1}^3$ or $A_{4,1}^6$.
Since $\g=B_2^6$, we have $\s_{(i)}^a\cong B_2^2$ if $\s_{(i)}\cong C_4$
and $\s_{(i)}^a\cong B_2$ if $\s_{(i)}\cong A_4$.
By using $\mathrm{rank}(\g)=6$, we see that $\tilde{V}_1\cong A_{4,1}^{6}$.
As a~result,
$W$ is isomorphic to $V_{N(\k)}$.

\noindent
{\bf (A.3) Case of $n\geq 3$, $n|12$.}
Since $\g=B_n^{12/n}$, it follows by Propositions~\ref{sec:decomp1}
 and~\ref{sec:classify1} that
$\s_{(i)}=B_n$, $A_{2n}$ or $D_{2n+1}$.
As $(2n+1)|h_i^\vee$, we have $\s_{(i)}=A_{2n}$, which forces that
$W_1\cong A_{2n,1}^{12/n}$.
Hence, $W\cong V_{N(\k)}$.

By combining (A.1)--(A.3), we see that $W\cong V_{N(\k)}$
for each $n|12$.

\noindent
{\bf (B)} Let $(\g,\k)=(D_{2n,2}^{4/n}B_{n,1}^{8/n},A_{4n-1,1}^{4/n}D_{2n+1,1}^{4/n})$, where
$n|4$.
It follows by~\eqref{eqn:montague}
that $W_1$ has dimension $96n+24$.
Then
$h_i^\vee/\ell_i=(\dim W_1-24)/24=4n$.

\noindent
{\bf (B.1) Case of $n=1$.}
Since $4|h_i^\vee$ and $\dim\, W_1=120$, it follows that $\mathfrak{s}_{(i)}$ is of type
$A_3$, $A_7$, $C_3$, $C_7$, $D_5$, $D_7$, $E_6$ or $G_2$.
By applying Proposition~\ref{sec:classify1} to $W_1$ and $a$, we see that $\mathfrak{s}_{(i)}$ must be $A_3$ or $G_2$.
Since $\dim W_1=120$, it follows that $W_1\cong A_3^8$,
and hence, $W\cong V_{N(\k)}$.

\noindent
{\bf (B.2) Case of $n=2$.}
It follows that $\mathfrak{s}_{(i)}$ has the type
$A_7$, $C_7$, $D_5$, or $D_9$.
Since $\dim\,W_1=216$, we have
$W_1=A_{7,1}D_{9,2}$ or $A_{7,1}^2D_{5,1}^2$.
Suppose that $W_1=A_{7,1}D_{9,2}$. It then follows that $A_7^a=D_4$.
Therefore, we have $D_9^a=D_4B_2^2$, which contradicts Proposition~\ref{sec:classify1}.
Hence,  $W_1$ must have the type $A_{7,1}^2D_{5,1}^2$.
As a result,
$W$ is isomorphic to $V_{N(\k)}$.

\noindent
{\bf (B.3) Case of $n=4$.}
Then $\mathfrak{s}_{(i)}$ has the type $A_{15}$ or $D_9$, which shows that
$W_1=A_{15,1}D_{9,1}$. As a result,
$W$ is isomorphic to $V_{N(\k)}$.

\noindent
{\bf (C)} Let $(\g,\k)=(D_{2n+1,2}^{4/n}A_{2n-1,1}^{4/n}, A_{4n+1,1}^{4/n}X_{n,1})$,
 where $n|4$.
Since $\dim\g=24 (2 n+1)$, it follows by~\eqref{eqn:montague}
that $W_1$  has dimension $24(4n+3)$.
Then we have
$h_i^\vee/\ell_i=(\dim W_1-24)/24=4n+2$.

\noindent
{\bf (C.1) Case of $n=1$.}
Since $\dim(W_1) =168$
and $h_i^\vee/k_i=6$,
it follows that $\mathfrak{s}_{(i)}$ has the type $A_{11}$, $D_7$, $E_6$, $C_5$, $A_5$, $D_{4}$, or $E_7$.
Since $\g=D_3^4A_1^4$, it follows by Proposition~\ref{sec:classify1} that
$\s_{(i)}\cong A_5$ or $D_4$, which forces that
$W_1=A_{5,1}^4D_{4,1}$ or $D_{4,1}^6$.
Since $D_3\subset \g$, we see that $A_5\subset W$, and hence
$W_1$ has the type $A_{5,1}^4D_{4,1}$.
Therefore,
$W\cong V_{N(\k)}$.

\noindent
{\bf (C.2) Case of $n=2$.}
We see that $\mathfrak{s}_{(i)}$ has the type
$A_9$, $C_9$, $D_6$, $D_{11}$ or $E_8$, which shows that
$W_1=D_{6,1}^4$ or $A_{9,1}^2D_{6,1}$.
Since $D_5\subset \g$, we have $W_1\cong A_{9,1}^2D_{6,1}$,
and hence,
$W\cong V_{N(\k)}$.

\noindent
{\bf (C.3) Case of $n=4$.}
Then $\mathfrak{s}_{(i)}$ has the type one of
$A_{17}$, $D_{10}$, or $E_7$, which forces that
$W_1=A_{17,1}E_{7,1}$ or $D_{10,1}E_{7,1}^2$.
 Since $\g=D_9A_7$, it follows by Proposition~\ref{sec:classify1} that
 the multiplicity of the ideal $E_7$ in
$W_1$ is less than $2$.
Therefore, $W_1$ has the type $A_{17,1}E_{7,1}$. As a result,
$W$ is isomorphic to $V_{N(\k)}$.

\noindent
{\bf (D)} Let $(\mathfrak{g},\mathfrak{f})=(C_{4,1}^{4}, E_{6,1}^{4})$.
By \eqref{eqn:montague}, we know that $\dim\,W_1=312$
and $h_i^\vee/\ell_i=12$.
It follows that $\mathfrak{s}_{(i)}$ has the type
$C_{11}$, $A_{11}$, $D_7$ or $E_6$, which shows that
$W_1=A_{11,1}D_{7,1}E_{6,1}$ or $E_{6,1}^{4}$.
Since  $A_{11}^a$ is isomorphic to $D_6$ or $C_6$, it follows by Proposition \ref{sec:classify1} and Corollary \ref{sec:decomp2} that $W_1\not\cong A_{11,1}D_{7,1}E_{6,1}$. Thus, $W_1$ has the type $E_{6,1}^4$.
Hence,
$W$ is isomorphic to $V_{N(\k)}$.

\noindent
{\bf (E)} Let $(\mathfrak{g},\mathfrak{f})=(D_{6,2}B_{3,1}^2C_{4,1}, A_{11,1}D_{7,1}E_{6,1})$. By the same argument, we have $\dim(W_1)=312$ and $h_i^\vee/\ell_i=(\dim W_1-24)/24=12$.
It follows that $\mathfrak{s}_{(i)}$ has the type
$A_{11}$, $C_{11}$, $D_7$ or $E_6$, which forces that
$W_1=E_{6,1}^4$ or $A_{11,1}D_{7,1} E_{6,1}$.
Since $\g=D_6B_3^2C_4$, it follows by Proposition~\ref{sec:classify1} that the multiplicity of the ideal $E_6$ in $W_1$ is less than $2$.
Thus, $W_1$ has the type $A_{11,1}D_{7,1} E_{6,1}$, and thus,
$W$ is isomorphic to $V_{N(\k)}$.
\qed


\smallskip

To summarize, we have proved that  there exists a semisimple element $h\in V_1$ such that{\rm:}
{\rm(1)} $\sigma_h$ is an automorphism of $V$ of order $2$,
and the pair $(V,\sigma_h)$ satisfies the orbifold condition{\rm;}
{\rm(2)} The holomorphic VOA $\tilde{V}(\sigma_h)=V^{\sigma_h}\oplus V^{\mathrm{T}}(\sigma_h)_{\Z}$ is isomorphic to $V_{N(\k)}${\rm;}
{\rm(3)} $\mathfrak{g}^{\sigma_h}=\g$. Taking $U=V_{N(\k)}$, we can obtain  Theorem~\ref{mainresult} by Theorems \ref{sec:unique1}, \ref{main2}.

\subsection*{Acknowledgements}
The authors would like to thank Hiroki Shimakura for helpful discussions.
Some computations are done by using the computer algebra system Sage.

\end{document}